\documentclass[12pt]{amsart}
\usepackage{amscd,amsmath,amsthm,amssymb,graphics}

\usepackage{pgf,tikz}
\usetikzlibrary{arrows}

\usepackage[bookmarks=true]{hyperref}

\def\NZQ{\Bbb}               % the font for N,Z,Q,R,C

\def\KK{{\NZQ K}}

\def\frk{\frak}               % font for "Fraktur"

\def\Phi{{\frk n}}
\def\Phi{{\frk N}}
\def\MF{{\mathcal F}}

\def\MS{{\mathcal S}}

\def\opn#1#2{\def#1{\operatorname{#2}}} % to make operators
\opn\chara{char} \opn\length{\ell} \opn\pd{pd} \opn\rk{rk}
\opn\projdim{proj\,dim} \opn\injdim{inj\,dim} \opn\rank{rank}
\opn\depth{depth} \opn\grade{grade} \opn\height{height}
\opn\embdim{emb\,dim} \opn\codim{codim}

\opn\Tr{Tr} \opn\bigrank{big\,rank}
\opn\superheight{superheight}\opn\lcm{lcm}
\opn\trdeg{tr\,deg}%\emph{
\opn\reg{reg} \opn\lreg{lreg} \opn\ini{in} \opn\lpd{lpd}
\opn\size{size}\opn\bigsize{bigsize}
\opn\cosize{cosize}\opn\bigcosize{bigcosize}
\opn\sdepth{sdepth}\opn\sreg{sreg}
\opn\link{link}\opn\fdepth{fdepth}
\opn\index{index}
\opn\index{index}
\opn\indeg{indeg}
\opn\N{N}
\opn\SSC{SSC}
\opn\SC{SC}
\opn\del{del}
\opn\div{div} \opn\Div{Div} \opn\cl{cl} \opn\Cl{Cl}
\opn\Spec{Spec} \opn\Supp{Supp} \opn\supp{supp} \opn\Sing{Sing}
\opn\Ass{Ass} \opn\Min{Min}\opn\Mon{Mon} \opn\dstab{dstab} \opn\astab{astab}
\opn\Syz{Syz}
\opn\reg{reg}
\opn\Ann{Ann} \opn\Rad{Rad} \opn\Soc{Soc}
\opn\Im{Im} \opn\Ker{Ker} \opn\Coker{Coker} \opn\Am{Am}
\opn\Hom{Hom} \opn\Tor{Tor} \opn\Ext{Ext} \opn\End{End}
\opn\Aut{Aut} \opn\id{id}

\opn\nat{nat}
\opn\pff{pf}%   \pf exists already
\opn\Pf{Pf} \opn\GL{GL} \opn\SL{SL} \opn\mod{mod} \opn\ord{ord}
\opn\Gin{Gin} \opn\Hilb{Hilb}\opn\sort{sort}
\opn\initial{init}
\opn\ende{end}
\opn\height{height}
\opn\type{type}
\opn\aff{aff} \opn\con{conv} \opn\relint{relint} \opn\st{st}
\opn\lk{lk} \opn\cn{cn} \opn\core{core} \opn\vol{vol}
\opn\link{link} \opn\star{star}\opn\lex{lex}\opn\Mon{Mon}\opn\Min{Min}
\opn\gr{gr}

\def\pot#1#2{#1[\kern-0.28ex[#2]\kern-0.28ex]}

\opn\dirlim{\underrightarrow{\lim}}
\opn\inivlim{\underleftarrow{\lim}}
\let\union=\cup

\let\Sect=\bigcap

\let\to=\rightarrow

\def\Implies{\ifmmode\Longrightarrow \else
        \unskip${}\Longrightarrow{}$\ignorespaces\fi}
\def\implies{\ifmmode\Rightarrow \else
        \unskip${}\Rightarrow{}$\ignorespaces\fi}
\def\iff{\ifmmode\Longleftrightarrow \else
        \unskip${}\Longleftrightarrow{}$\ignorespaces\fi}

\let\:=\colon
\newtheorem{Theorem}{Theorem}[section]
 
 \newtheorem{Corollary}[Theorem]{Corollary}

\let\epsilon\varepsilon
\let\kappa=\varkappa
\def\qed{\ifhmode\textqed\fi
      \ifmmode\ifinner\quad\qedsymbol\else\dispqed\fi\fi}
\def\textqed{\unskip\nobreak\penalty50
       \hskip2em\hbox{}\nobreak\hfil\qedsymbol
       \parfillskip=0pt \finalhyphendemerits=0}
\def\dispqed{\rlap{\qquad\qedsymbol}}

\opn\dis{dis}
\def\pnt{{\raise0.5mm\hbox{\large\bf.}}}

\opn\Lex{Lex}
\opn\set{set}

\begin{document}

 \title{Simplicial complexes of whisker type}

\author {Mina Bigdeli, J\"urgen Herzog, Takayuki Hibi and  Antonio Macchia}

\address{Mina Bigdeli, Department  of Mathematics,  Institute for Advanced Studies in Basic Sciences (IASBS),
45195-1159 Zanjan, Iran} \email{m.bigdelie@iasbs.ac.ir}

\address{J\"urgen Herzog, Fachbereich Mathematik, Universit\"at Duisburg-Essen, Campus Essen, 45117
Essen, Germany} \email{juergen.herzog@uni-essen.de}

\address{Takayuki Hibi, Department of Pure and Applied Mathematics, Graduate School of Information Science and Technology,
Osaka University, Toyonaka, Osaka 560-0043, Japan}
\email{hibi@math.sci.osaka-u.ac.jp}

\address{Antonio Macchia, Fachbereich Mathematik und Informatik, Philipps-Universit\"at Marburg, Hans-Meerwein-Strasse 6, 35032 Marburg, Germany} \email{macchia.antonello@gmail.com}

\thanks{The paper was written while the first and the fourth author were  visiting the Department of Mathematics of University Duisburg-Essen. They want to express their thanks for the hospitality.}

 \begin{abstract}
Let $I\subset K[x_1,\ldots,x_n]$ be  a zero-dimensional monomial ideal, and $\Delta(I)$ be the simplicial complex whose Stanley--Reisner ideal is the polarization of $I$. It follows from a result of   Soleyman Jahan  that $\Delta(I)$ is shellable. We give a new short proof of this fact by providing an explicit shelling.  Moreover, we show that  $\Delta(I)$ is even vertex decomposable.
The ideal $L(I)$, which is defined to be the Stanley--Reisner ideal of the Alexander dual of  $\Delta(I)$, has a linear resolution which is  cellular
and supported on a regular CW-complex. All powers of $L(I)$ have a linear resolution. We compute $\depth L(I)^k$ and show that $\depth L(I)^k=n$ for all $k\geq n$.
 \end{abstract}


 \maketitle

\section*{Introduction}

Graphs with whiskers have first been considered by Villarreal in \cite{V}. They all share  the nice property that they are Cohen-Macaulay. Various extensions of this concept and generalizations of his result have been considered in the literature, see  \cite{CN}, \cite{VTV}, \cite{HHKO} and \cite{F}. The edge ideal of a  whisker graph is obtained as the polarization of a monomial ideal $I\subset S$, where $S=K[x_1,\ldots,x_n]$ is the polynomial ring over a field $K$,  $I$ is generated in degree $2$ and  $\dim S/I=0$. In particular, $I$ contains the squares $x_1^2,\ldots,x_n^2$. More generally, given a simplicial complex $\Gamma$, the whisker complex $W(\Gamma)$ is studied in \cite{HH1}. Its facet ideal is the polarization of a monomial ideal in $S$ which contains all the $x_i^2$.  In \cite{HH1} the Stanley--Reisner ideal of the Alexander dual of the independence complex of $W(\Gamma)$  is called the face ideal of a simplicial complex.  Such face ideals, under a different name,  have also been studied in
\cite{LL}. In \cite{HH1}, as well as in \cite{LL}, it is shown that face ideals have a linear resolution. In  \cite{LL} Loiskekoski even shows that all powers of a face ideal have a linear resolution.

In the present paper we generalize the above mentioned results by considering the polarization of any monomial ideal $I\subset S$ with $\dim S/I=0$. The simplicial complex $\Theta(I)$, whose facet ideal  coincides with the polarization $I^\wp$ of $I$, is called of {\em whisker type} -- the whiskers being the simplices corresponding to the polarization of the pure powers contained in $I$. The independence complex of $\Theta(I)$, denoted $\Delta(I)$,   is characterized by the property that the Stanley--Reisner ideal $I_{\Delta(I)}$ coincides with $I^\wp$. Note that $F\in \Delta(I)$  if and only if $F$ does not contain any facet of $\Theta(I)$.

Given an arbitrary monomial ideal $I\subset S$, a multicomplex is associated with $I$, as defined by Popescu and the second author in \cite{HP}.  Soleyman Jahan defines in\break \cite[Proposition 3.8]{S} a bijection between the facets of the multicomplex given by $I$ and the facets of the simplicial complex  associated with $I^\wp$. In Theorem~\ref{min}  we present a short proof of this bijection when $\dim S/I=0$, by using multiplicity theory. This result allows us  to describe in Corollary~\ref{facets} the facets of $\Delta(I)$.  By applying the Eagon--Reiner Theorem it is then shown in Corollary~\ref{linear} that the ideal $L(I)$ has a linear resolution, where $L(I)$ is generated by the monomials $x_{1,a_1+1} \cdots x_{n,a_n+1}$ for which  $x_1^{a_1}\cdots x_n^{a_n}$ is a monomial in $S$  not belonging to $I$.

In the case that $\dim S/I=0$, the case we consider here, the corresponding multicomplex is pretty clean, see  \cite{HP}. Soleyman Jahan showed in \cite[Theorem 4.3]{S} that if $I$ defines a pretty clean multicomplex, then the simplicial complex associated with  $I^\wp$ is clean, which, by a theorem of Dress \cite{D}, implies that the simplicial complex attached to $I^\wp$ is shellable. Applied to our situation it follows that $\Delta(I)$ is shellable. We give a direct proof of this fact by showing that $L(I)$ has linear quotients. This provides an explicit shelling of $\Delta(I)$, and as  a side result we obtain a formula for the  Betti numbers of $L(I)$ in terms of the $h$-vector of $S/I$, see Corollary~\ref{betti}.  We conclude Section~2  with Corollary~\ref{cellular}, where it is shown that  the minimal  graded free resolution of $L(I)$ is cellular
and supported on a regular CW-complex. The proof is based on a result of  Dochtermann and Mohammadi\break \cite[Theorem~3.10]{DM}, who showed that the minimal graded free  resolution of any ideal with regular decomposition function, as defined in \cite{HT}, have such nice cellular structure.

In Section~3 we show that $\Delta(I)$ is not only shellable but even vertex decomposable. This was already  known  for whisker graphs (see \cite[Theorem~4.4]{DE}).
Finally in Section~4 we prove that all powers of $L(I)$ have linear quotients, see Theorem~\ref{powers}. Analyzing the linear quotients,  the depth function $f(k)=\depth S/L(I)^k$ can be computed. In Corollary~\ref{depth} a formula for the depth function   is given and $\lim_{k\to \infty} \depth S/L(I)^k$ is determined.

\section{The facets of the independence complex of a whisker type simplicial complex}

Throughout this paper $S$ denotes the polynomial ring $K[x_1,\ldots,x_n]$ and $I\subset S$ a monomial ideal with $\dim S/I=0$, unless otherwise stated. The (finite) set of monomials in $S$ which  belong to $S$ but not to $I$ will be denoted by $\Mon(S\setminus I)$. For an arbitrary monomial ideal $I$, we denote by $G(I)$ the unique minimal set of monomial generators of $I$. We will consider the polarization of $I$, denoted $I^\wp$. The polynomial ring in which $I^\wp$ is defined will be denoted by $S^\wp$.

In the following theorem  (cf. \cite[Proposition 3.8]{S}) we determine the set $\Min(I^\wp)$ of  minimal prime ideals of $I^\wp$.

\begin{Theorem}
\label{min}
Let $I\subset S$ be a monomial ideal with $\dim S/I=0$. The map $\phi$ which assigns to each monomial $u=x_1^{a_1}\cdots x_n^{a_n}\in S\setminus I$ the monomial prime ideal $\phi(u)=(x_{1,a_1+1}, \ldots, x_{n,a_n+1})\subset S^\wp$,  establishes a bijection between $\Mon(S\setminus I)$ and $\Min(I^\wp)$.
\end{Theorem}

\begin{proof}
We first observe that $\phi(\Mon(S\setminus I)) \subset S^\wp$. Indeed, since $\dim S/I=0$, there exists for each $1\leq i\leq n$ an integer $b_i>0$ such that $x_i^{b_i}\in I$ and $x_i^{b_i-1}\notin I$. It follows that  $S^\wp$ is the polynomial ring in the variables $x_{i,1},  \ldots, x_{i,b_i}$ with $1\leq i\leq n$.  Now let $u=x_1^{a_1}\cdots x_n^{a_n}\in \Mon(S\setminus I)$. Then $a_i<b_i$ for all $i$, and this implies that $\phi(u)\in S^\wp$.

Next we show that $\phi(\Mon(S\setminus I))\subset \Min(I^\wp)$. In fact, let $u=x_1^{a_1}\cdots x_n^{a_n}$ be an element in $\Mon(S\setminus I)$, and let $v\in G(I)$. We claim  that there exists an integer $i$ such that $x_{i,a_i+1}$ divides $v^\wp$, where $v^\wp$ is the polarization of $v$. From this claim it follows that $I^\wp\subset \phi(u)$. Since $\height I^\wp=\height I=n$ and since $\height \phi(u)=n$, we then see that $\phi(u)$ is in fact a minimal prime ideal of $I^\wp$.

Let $v=x_1^{b_1}\cdots x_n^{b_n}$. In order to prove the claim, note that $v^\wp=\prod_{i=1}^n(\prod_{j=1}^{b_i}x_{i,j})$. Since $v$ does not divide $u$, there exists an integer $i$ such that $b_i>a_i$. Therefore, $x_{i,a_i+1}$ divides $v^\wp$, as desired.

Clearly, $\phi$ is injective. We will show that $|\Mon(S\setminus I)|=|\Min(I^\wp)|$.  This will then imply that $\phi\: \Mon(S\setminus I)\to \Min(I^\wp)$ is bijective. In order to see that these two sets have the same cardinality we observe that the multiplicity $e(S/I)$ of $S/I$ is equal to the length $\ell(S/I)$ of $S/I$, because  $\dim S/I=0$, see \cite[Corollary 4.7.11(b)]{BH}.  Since $\ell(S/I)=\dim_KS/I$ and since  the elements of $\Mon(S\setminus I)$ form a $K$-basis of $S/I$, we see that $e(S/I)=\dim_KS/I=|\Mon(S \setminus I)|$. On the other hand, since $S/I$ is isomorphic to $S^\wp/I^\wp$ modulo a regular sequence of linear forms \cite[Proposition~1.6.2]{HHBook}, and since $S^\wp/I^\wp$ is reduced and equidimensional, \cite[Corollary 4.7.8]{BH} implies that $e(S/I)=e(S^\wp/I^\wp)=|\Min(I^\wp)|$.
\end{proof}

We denote by $\Delta(I)$ the simplicial complex whose Stanley-Reisner ideal is $I^\wp$. We view the variables $x_{i,j}\in S^\wp$ as the vertices of $\Delta(I)$. As an immediate consequence of Theorem~\ref{min} we obtain

\begin{Corollary}
\label{facets}
Let $\MS$  be the set of variables of $S^\wp$. Then $F\subset \MS$ is a facet of $\Delta(I)$ if and only if there exists $x_1^{a_1}\cdots x_n^{a_n}\in  \Mon(S \setminus I)$ such that $$F=\MS\setminus \{x_{1,a_1+1}, \ldots, x_{n,a_n+1}\}.$$
\end{Corollary}

Since $\Delta(I)$ is Cohen--Macaulay,  the Eagon--Reiner Theorem \cite{EG} (see also \cite[Theorem~8.1.9]{HHBook}) implies that $I_{\Delta(I)^\vee}$ has a linear resolution. Here $\Delta(I)^\vee$  denotes the Alexander dual of $\Delta(I)$. Recall that, if $\Delta$ is an arbitrary simplicial complex on the vertex set $[n]=\{1,\ldots,n\}$ and  $I_\Delta=\Sect_{F} P_F$ where $P_F=(x_i\:\; i\in F)$, then $I_{\Delta^{\vee}}$ is generated by the monomials $u_F$ where $u_F=\prod_{i\in F}x_i$. These facts applied to our case yield

\begin{Corollary}
\label{linear}
The ideal $L(I)$ generated by the monomials $x_{1,a_1+1} \cdots x_{n,a_n+1}$, with $x_1^{a_1}\cdots x_n^{a_n}\in  \Mon(S\setminus I)$, has a linear resolution.
\end{Corollary}

In the following we consider the special case that $x_i^2\in I$ for all $i$. In that case all other generators of $I$ are square-free. In simplified notation, the polarization $I^\wp$ of $I$  is generated by the square-free monomials in $I$ and by the monomials $x_iy_i$  for $i=1,\ldots,n$.

Let $\Gamma$ be the simplicial complex with $I(\Gamma)=J$ and $W(\Gamma)$ be the simplicial complex with $I(W(\Gamma))= (J,x_1y_1,\ldots,x_ny_n)$. The edges of $W(\Gamma)$ corresponding to the $x_iy_i$  are called the {\em whiskers} of $W(\Gamma)$ and $W(\Gamma)$  is called the {\em whisker complex} of $\Gamma$.

Given a simplicial complex $\Sigma$, the independence complex $\Lambda$ of $\Sigma$ is the simplicial complex such that $I_\Lambda=I(\Sigma)$. Notice that $F\in \Lambda$ if and only if no face of $\Sigma$ is contained in $F$.

As a special case of Theorem~\ref{min} we recover \cite[Theorem 1.1]{HH1}.

\begin{Corollary}
\label{newproof}
Let $\Gamma$ be a simplicial complex on the vertex set $[n]$, $I'=I(\Gamma)$ the facet ideal of $\Gamma$  and $W(\Gamma)$ its whisker complex. Let $I=(I', x_1^2,x_2^2,\ldots,x_n^2)$. Then $\Delta(I)$ is the independence complex of $W(\Gamma)$ and $L(I)$ is generated by the monomials $\prod_{i\in [n]\setminus F}x_i\prod_{i\in  F}y_i$ with $F\in \Delta$, where $\Delta$ is the independence complex of $\Gamma$.
\end{Corollary}

\section{Linear quotients}

Let $I\subset S$ be a monomial ideal with $\dim S/I=0$. The main purpose of this section is to show that $L(I)$ not only has a linear resolution, but even has linear quotients.

\begin{Theorem}
\label{quotients}
The ideal $L(I)$ has linear quotients.
\end{Theorem}

\begin{proof}
Let $u,v\in G(L(I))$, $u=x_{1,a_1+1} \cdots x_{n,a_n+1}$ and $v=x_{1,b_1+1} \cdots x_{n,b_n+1}$. We set $u\leq v$  if $a_i\leq  b_i$ for all $i$, and extend this partial order to a total order on $G(L(I))$. We claim that, with respect to this total order of the monomial generators of $L(I)$, the ideal $L(I)$ has linear quotients. Indeed, let $x_{1,a_1+1} \cdots x_{n,a_n+1}$ be the largest element in $G(L(I))$. Then $u=x_1^{a_1} \cdots x_n^{a_n}\in \Mon(S\setminus I)$ and $x_iu \in I$ for all $i$. Set $I' = I + (u)$. Then the polarization $(I')^\wp$ of $I'$ is equal to  $I_{\Delta(I')}$. Notice that $L(I') \subset L(I)$ and $\ell(S/I') < \ell(S/I)$. In particular, $L(I) = (L(I'),x_{1,a_1+1} \cdots x_{n,a_n+1})$. Arguing by induction  on the length, we may assume that $L(I')$ has linear quotients. Thus we just need to compute the colon ideal $Q=L(I') : x_{1,a_1+1} \cdots x_{n,a_n+1}$. We claim that
\begin{eqnarray}
\label{q}
Q=(x_{1,1},x_{1,2},\ldots,x_{1,a_1},x_{2,1},\ldots,x_{2,a_2},\ldots, x_{n,1},\ldots,
x_{n,a_n}).
\end{eqnarray}
Suppose that $j \in \{1,\dots,a_i\}$ for some $i$. Then $x_1^{a_1} \cdots x_i^{j-1} \cdots x_n^{a_n} \in \Mon(S \setminus I)$ and $$\phi(x_1^{a_1} \cdots x_i^{j-1} \cdots x_n^{a_n}) = x_{1,a_1+1} \cdots x_{i,j} \dots x_{n,a_n+1} \in L(I').$$ It follows that $x_{i,j} \in Q$.

On the other hand, the elements $v/ \gcd(v, x_{1,a_1+1} \cdots x_{n,a_n+1})$ with $v\in G(L(I'))$ generate $Q$, see for example \cite[Proposition~1.2.2]{HHBook}. In fact, let $v\in G(L(I'))$. Then $v=x_{1,c_1+1}\cdots x_{n,c_n+1}$ and $x_1^{c_1}\cdots x_n^{c_n}\in \Mon(S\setminus I')$. There exists $i$ such that $c_i < a_i$ because $x_iu \in I$ for all $i$. Hence  $x_{i,c_i+1}$ does not divide $x_{1,a_1+1} \cdots x_{n,a_n+1}$, and  therefore $x_{i,c_i+1}$ divides $v/ \gcd(v, x_{1,a_1+1} \cdots x_{n,a_n+1})$. Since $c_i+1\leq a_i$, the desired conclusion follows.
\end{proof}

\begin{Corollary}
\label{betti}
For every $i \geq 0$,
\[
\beta_i (S^\wp/L(I)) = \sum_{j \geq 0} h_j \binom{j}{i-1},
\]
where $h_j = h_j(S/I)$ is the $j$-th component of the $h$-vector of $S/I$. In particular, $\projdim S^\wp/L(I)=\max\{\deg u\: u\in \Mon(S\setminus I)\}+1$.
\end{Corollary}

\begin{proof}
As in the previous proof, let $u=x_1^{a_1} \cdots x_n^{a_n}\in \Mon(S \setminus I)$ with  $x_iu \in I$ for all $i$. Set $I' = I + (u)$, and  consider the short exact sequence $$0 \rightarrow L(I)/L(I') \rightarrow S^\wp/L(I') \rightarrow S^\wp/L(I) \rightarrow 0.$$ Notice that $L(I)/L(I') \cong S^{\wp}/Q (-n)$ with $Q$ as in (\ref{q}). Hence its minimal free resolution is the Koszul complex $\KK$ on the variables $x_{i,j}$  with $x_{i,j}\in G(Q)$. Thus the   minimal free resolution of $S^\wp/L(I)$ can be obtained as a mapping cone of $\KK$  and the minimal free resolution of $S^\wp/L(I')$. Therefore $\beta_0 (S^\wp/L(I)) = \beta_0 (S^\wp/L(I'))$, and for $i \geq 1$ we obtain
\begin{eqnarray*}
\beta_i (S^\wp/L(I)) &=& \beta_i(S^\wp/L(I')) + \rank (K_{i-1}) = \beta_i(S^\wp/L(I')) + \binom{\deg u}{i-1} \\
&=& \sum_{u \in \Mon(S \setminus I)} \binom{\deg u}{i-1} = \sum_{j \geq 0} h_j \binom{j}{i-1}.
\end{eqnarray*}
\end{proof}

It is easily seen that the geometric realization of $\Delta(I)$ is a sphere if $I$ is a complete intersection, and a ball otherwise. Both topological spaces admit shellable triangulations, though in general not all triangulations of these spaces are shellable, see \cite{Ru} and \cite{Li}. However, due  to Theorem~\ref{quotients} we have

\begin{Corollary}
\label{shelling}
The simplicial complex  $\Delta(I)$ is shellable.
\end{Corollary}

As a further consequence of Theorem~\ref{quotients} we have

\begin{Corollary}
\label{cellular}
The graded minimal free resolution of $L(I)$ is cellular
and supported on a regular CW-complex.
\end{Corollary}

\begin{proof}
Since $L(I)$ has linear quotients we may apply \cite[Theorem 3.10]{DM} and only need to show that $L(I)$ admits a regular decomposition function. In order to explain this, let $J=(u_1,\ldots,u_m)$ be an ideal with linear quotients  with respect to the given order of the generators. The {\em decomposition function} of $J$  (with respect to the given order of the generators of $J$) is the map $b\:  \Mon(J) \to G(J)$ with $b(u) = u_j$, where $j$ is the smallest number such that $u\in (u_1,\ldots,u_j)$. For each $u_j\in G(J)$, let  $\set(u_j)$ be the set of all $x_i$ such that $x_iu_j\in (u_1,\ldots,u_{j-1})$. According to \cite{HT}, the decomposition function $b$ is called {\em regular}, if $\set(b(x_iu_j))\subset \set(u_j)$ for all $u_j\in G(J)$ and all $x_i\in \set(u_j)$.

Now let $u\in G(L(I))$, $u=x_{1,a_1+1} \cdots x_{n,a_n+1}$. By (\ref{q})  we have
\[
\set(u)=\{x_{1,1},x_{1,2},\ldots,x_{1,a_1},x_{2,1},\ldots,x_{2,a_2},\ldots, x_{n,1},\ldots,
x_{n,a_n}\}.
\]
Let $x_{i,j}\in \set(u)$. Then $b(x_{i,j}u)=x_{i,j}(u/x_{i,a_i+1})$, and so
\[
\set(b(x_{i,j}u))=\set(u)\setminus\{x_{i,j+1},\ldots,x_{i,a_i}\}\subset \set(u),
\]
as desired.
\end{proof}

\section{Vertex decomposability}

In \cite[Theorem~4.4]{DE} it was shown that for any graph, the independence complex of its whisker graph is vertex decomposable. Here we extend this result by showing that $\Delta(I)$ is vertex decomposable for any monomial ideal $I$ with $\dim S/I=0$.
Recall that  a simplicial complex $\Delta$ is called {\em vertex decomposable} if $\Delta$  is
a simplex, or $\Delta$ contains a vertex $v$ such that
\begin{itemize}
\item[(i)]  any facet of $\del_{\Delta}∆(v)$  is a facet of $\Delta$, and
\item[(ii)] both $\del_{\Delta}(v)$ and $\link_{\Delta}(v)$  are vertex decomposable.
\end{itemize}
Here $\link_\Delta(v)=\{G\in \Delta\:\; v\not\in G \text{ and } G\union \{v\}\in \Delta\}$ is the {\em link} of $v$ in $\Delta$ and $\del_\Delta(v) =\{G\in \Delta\:\; v\not\in G \}$ is the {\em deletion} of $v$ from $\Delta$.

\medskip
A vertex $v$  which satisfies condition (i) is called a {\em shedding vertex} of $\Delta$.

\medskip
For the proof of the next result we observe the following fact: let $\Delta$ be a simplicial complex and $v$ a vertex not belonging to $\Delta$. The {\em cone} of $v$ over $\Delta$, denoted by $v\ast \Delta$, is the simplicial complex such that $\MF(v\ast \Delta)=\{\{v\}\union F\:\;  F\in \MF(\Delta)\}$. If $\Delta$ is vertex decomposable, then $v\ast \Delta$ is again vertex decomposable (with respect to the same shedding vertex).

\begin{Theorem}
\label{decomposable}
Let $I$ be a monomial ideal in $S=K[x_1, \ldots, x_n]$ with $\dim S/I=0$. Then $\Delta(I)$ is vertex decomposable.
\end{Theorem}

\begin{proof}
By assumption, for each $1\leq i\leq n$ there exists $b_i \geq 1$ such that $x_i^{b_i}\in G(I)$. Then $\Delta(I)$ is a simplicial complex on $\mathcal{S}=\{x_{1,1},\dots, x_{1,b_1},\dots, x_{n,1},\dots,x_{n,b_n}\}$. We proceed by induction on $\sum_{i=1}^n b_i$. If $\sum_{i=1}^n b_i = n$, then $I=(x_1,\dots,x_n)$, which is a trivial case. Suppose that $\sum_{i=1}^n b_i > n$. Hence we may assume $b_n > 1$.

We first show that the vertex $x_{n,1}$  is a shedding vertex of $\Delta(I)$. Clearly,
$$\del_{\Delta(I)} (x_{n,1}) = \{ F : F \in \Delta(I), x_{n,1} \notin F \} \cup \{ F \setminus \{x_{n,1}\} : F \in \Delta(I), x_{n,1} \in F \}.$$
Obviously, any facet  of $\del_{\Delta(I)} (x_{n,1})$ with  $x_{n,1} \notin F$ is a facet of $\Delta(I)$.
On the other hand,  if we consider $F \setminus \{x_{n,1}\}$ with $F \in \mathcal{F}(\Delta(I))$ and $x_{n,1} \in F$, then $F \setminus \{x_{n,1}\}$ is not a facet of $\del_{\Delta(I)} (x_{n,1})$. Indeed, since $F \in \mathcal{F}(\Delta(I))$, there exists $u \in \Mon(S \setminus I)$ such that $\phi(u) = P_{\mathcal{S}\setminus F}$. Let $t$ be the largest integer such that $x_n^t$ divides $u$. Then $x_{n,t+1}\in P_{\mathcal{S}\setminus F}$ and so $x_{n,j}\in F$ for all $j\neq t+1$. Since $x_{n,1}\in F$, we have $t+1\neq 1$. Let $u'=u/x_n^t$. Then $u'\in \Mon(S\setminus I)$ and $\phi(u')= P_{((\mathcal{S}\setminus F)\setminus \{x_{n,t+1}\}) \cup \{x_{n,1}\}}$. Thus $G=(F\setminus \{x_{n,1}\})\cup \{x_{n,t+1}\} \in \mathcal{F}(\Delta(I))$. Since $G\in \del_{\Delta(I)} (x_{n,1})$, the claim follows. Consequently, $\mathcal{F}(\del_{\Delta(I)} (x_{n,1})) = \{ F : F \in \mathcal{F}(\Delta(I)), x_{n,1} \notin F \}$  which implies that $x_{n,1}$ is a shedding vertex of $\Delta(I)$.

We now prove that $\del_{\Delta(I)} (x_{n,1})$ and $\link_{\Delta(I)}(x_{n,1})$ are vertex decomposable.

First we consider $\del_{\Delta(I)} (x_{n,1})$. Let $J_1$ be the ideal in $S$ with $\Mon(S \setminus J_1)=\{u:\ u\in \Mon(S\setminus I), x_n \text{ does not divide }  u\}$. Then $\Delta(J_1)$ is a simplicial complex on $\mathcal{S} \setminus \{x_{n,1},\dots,x_{n,b_n}\}$. By using Corollary~\ref{facets} we see that
$$
\del_{\Delta(I)}(x_{n,1}) =x_{n,b_n}\ast(x_{n,b_n-1}\ast(\cdots\ast(x_{n,2}\ast\Delta(J_1)))).
$$
Our induction hypothesis implies that $\Delta(J_1)$ is vertex decomposable, and hence  $\del_{\Delta(I)}(x_{n,1}) $ is vertex decomposable.

As for $\link_{\Delta(I)}(x_{n,1})$, let $\Gamma$ be the simplicial complex whose faces are obtained from the faces of $\link_{\Delta(I)}(x_{n,1})$ as follows: for every $F \in \link_{\Delta(I)}(x_{n,1})$, we replace each $x_{n,j} \in F$ by $x_{n,j-1}$. Hence $\Gamma$ is a simplicial complex on $\mathcal{S}\setminus \{x_{n,b_n}\}$ and\break $\Gamma \cong \link_{\Delta(I)}(x_{n,1})$. Let $J_2$ be the monomial ideal in $S$ such that $\Mon(S \setminus J_2) = \{u/x_n : u \in \Mon(S \setminus I), x_n \text{ divides } u\}$. Then Corollary~\ref{facets} implies that $\Gamma=\Delta(J_2)$, which is vertex decomposable by induction hypothesis. It follows that $\link_{\Delta(I)}(x_{n,1})$ is vertex decomposable, as desired.
\end{proof}

\section{Powers}

In this section we study the powers of $L(I)$. The main result is

\begin{Theorem}
\label{powers}
Let $I\subset S$ be a monomial ideal with $\dim S/I=0$. Then $L(I)^k$ has linear quotients for all $k$. In particular, all powers of $L(I)$ have a linear resolution.
\end{Theorem}

\begin{proof}
Any  $u\in L(I)^k$ can be written in the form $u=u_1'u_2'\cdots u_n'$, where $u_i'=x_{i,j(i)_1}x_{i,j(i)_2}\cdots x_{i,j(i)_k}$ for $i=1,\ldots,n$ with $j(i)_1\leq j(i)_2\leq \cdots \leq j(i)_k$. We define a partial order on  $G(L(I)^k)$ by setting $v\leq u$, if,  with respect to the  lexicographical  order, $u_i'\leq v_i'$ for all $i$, and we extend this partial order to a total order on the set of monomial generators of $L(I)^k$.

Now let $v,u\in L(I)^k$ with $v<u$. We need to show that there exists $w\in L(I)^k$ with $w<u$ such that $ w/\gcd(w,u)$ is of degree $1$ and such that $w/\gcd(w,u)$ divides $v/\gcd(v,u)$. Indeed, since $v<u$, there exists $i$ such that $u_i'< v_i'$ in the lexicographical order. Thus if $v_i'=x_{i,j'(i)_1}x_{i,j'(i)_2}\cdots x_{i,j'(i)_k}$ with $j'(i)_1\leq j'(i)_2\leq \cdots \leq j'(i)_k$, then there exists $\ell$ such that $j'(i)_s=j(i)_s$ for $s<\ell$ and $j'(i)_\ell<j(i)_\ell$. We let $w=w_1'w_2'\ldots w_n'$ with $w_t'=u_t'$  for $t\neq i$ and $w_i'=x_{i,j'(i)_\ell}(u_i'/x_{i,j(i)_\ell})$.

It is clear that $w<u$. Furthermore, $w\in G(L(I)^k)$. In fact, $u=u_1\cdots u_k$ with $u_i\in L(I)$,  and $x_{i,j(i)_\ell}$ divides one of these factors, say it divides $u_r$. Then  $\bar{u}_r=x_{i,j'(i)_\ell}(u_r/x_{i,j(i)_\ell})$ belongs to $L(I)$ since  $j'(i)_\ell<j(i)_\ell$,  and hence $w=u_1\cdots \bar{u}_r\cdots u_k$ belongs to $G(L(I)^k)$. Note further that $w/\gcd(w,u)=x_{i,j'(i)_\ell}$ and that $x_{i,j'(i)_\ell}$ divides $v/\gcd(v,u)$. This completes the proof.
\end{proof}

\begin{Corollary}
\label{depth}
For $i=1,\ldots,n$,  let $b_i$ be the smallest integer such that $x_i^{b_i}\in I$. Then
\begin{eqnarray*}
\depth S^\wp/L(I)^k\! =\!\sum_{i=1}^nb_i-\max\{\deg (\lcm(u_1,\ldots, u_k))\: u_1,\ldots, u_k\in \Mon(S\setminus I)\}-1.
\end{eqnarray*}
In particular, $\depth S^\wp/L(I)^k=n-1$ for all $k\geq n$, and $$\depth S^\wp/L(I)^k>\depth S^\wp/L(I)^{k+1},$$ as long as
$\depth S^\wp/L(I)^k>n-1$.
\end{Corollary}

\begin{proof}
In general, let  $J\subset K[x_1,\ldots,x_n]$ be  a graded ideal generated by a sequence $f_1, \ldots , f_s$ with linear quotients, and denote by $q_j(J)$ the minimal number of linear forms generating the ideal $(f_1,f_2, \ldots, f_{j-1}) : f_j$.
Then  $\depth K[x_1,\ldots,x_n]/J= n-q(J)-1$, where $q(J)=\max\{q_j(J)\ :\ 2\leq j\leq s\}$, see \cite[Formula (1)]{HH}.

We apply this formula to $S^\wp/L(I)^k$. Since the Krull dimension of $S^\wp$ is equal to $\sum_{i=1}^nb_i$, it remains to be shown that
\begin{eqnarray}
\label{aa}
q(L(I)^k)=\max\{\deg (\lcm(u_1,u_2,\ldots, u_k))\:\; u_1,u_2,\ldots,u_k\in \Mon(S\setminus I)\}.
\end{eqnarray}
To see this, let  $u_1,u_2,\ldots,u_k\in \Mon(S\setminus I)$ where  $u_j=x_{1}^{a_1(j)}\cdots x_{n}^{a_n(j)}$ for $j=1,\ldots,k$. Then
$u=u_1'u_2'\cdots u_n'$ with $u_i' =\prod_{j=1}^kx_{i, a_i(j)+1}$ is a generator of $L(I)^k$. We may assume that $u$ is the $j$-th element  in the given total order of the elements of $G(L(I)^k)$. As shown in the proof of Theorem~\ref{powers}, $q_j(L(I)^k)$ is the cardinality of the set
\[
\{x_{1,1},\ldots, x_{1,c_1}, x_{2,1},\ldots,x_{2,c_2},\ldots,x_{n,1},\ldots,x_{n,c_n}\},
\]
where $c_i =\max\{a_i(1),\ldots,a_i(k)\}$ for $i=1,\ldots,n$. If follows that $$q_j(L(I)^k)= \deg (\lcm(u_1,u_2,\ldots, u_k)), $$ and hence equation (\ref{aa}) follows.

Suppose now that $k\geq n$. Then we may choose $u_i=x_i^{b_i-1}$ for $i=1,\ldots,n$ and $u_i\in \Mon(S\setminus I)$ arbitrary for $i>n$, and obtain
$\deg(\lcm(u_1,u_2,\ldots u_k))=\sum_{i=1}^n(b_i-1)= \sum_{i=1}^nb_i-n$. Since this is the largest possible least common multiple of sequences of elements of $\Mon(S\setminus I)$, it follows that  $\depth S^\wp/L(I)^k=n-1$ for all $k\geq n$.

Finally, suppose that $\depth S^\wp/L(I)^k>n-1$.  Then the formula for $\depth S^\wp/L(I)^k$ implies that $\max\{\deg (\lcm(u_1,\ldots, u_k))\: u_1,\ldots, u_k\in \Mon(S\setminus I)\}<\sum_{i=1}^n(b_i-1).$

Let $x_1^{a_1}\cdots x_n^{a_n}=\lcm(u_1,\ldots, u_k)$ attain this maximal degree. Since $\sum_{i=1}^na_i<\sum_{i=1}^n(b_i-1)$, there exists an index $i$ such that $a_i<b_i-1$. Let $u_{k+1}=x_i^{b_i-1}$. Then $\deg(\lcm(u_1,\ldots, u_k,u_{k+1}))> \deg (\lcm(u_1,\ldots, u_k))$. Consequently,  $\depth S^\wp/L(I)^k>\depth S^\wp/L(I)^{k+1},$ as desired.
\end{proof}

\end{document}